\documentclass[11pt]{article}
\usepackage[cp1251]{inputenc}
\usepackage{amsmath,amsthm,mathrsfs}
\usepackage[english]{babel}

\usepackage{amsfonts,amstext,amssymb,verbatim,epsfig}
\usepackage{pgf,tikz}
\usepackage{float}
\usetikzlibrary{arrows}
\usepackage{hyperref}
\usepackage{cite}
\usepackage{epstopdf}
\usepackage{color}
\usepackage{transparent}
\usepackage{subfigure}

\def\R{{\mathbb{R}}}

\usepackage{xcolor}
\usepackage[normalem]{ulem}

\hypersetup{colorlinks=true,pdfborder=000,  citecolor  = blue, citebordercolor= {magenta}}
\hypersetup{linkcolor=black}
\makeatletter
\let\reftagform@=\tagform@
\def\tagform@#1{\maketag@@@{(\ignorespaces\textcolor{purple}{#1}\unskip\@@italiccorr)}}
\renewcommand{\eqref}[1]{\textup{\reftagform@{\ref{#1}}}}
\makeatother

\makeatletter

\DeclareUrlCommand\ULurl@@{%
  \def\UrlLeft{\uline\bgroup}%
  \def\UrlRight{\egroup}}
\def\ULurl@#1{\hyper@linkurl{\ULurl@@{#1}}{#1}}
\DeclareRobustCommand*\ULurl{\hyper@normalise\ULurl@}
\makeatother

\def\lessim{\ \lower4pt\hbox{$
		\buildrel{\displaystyle <}\over\sim$}\ }
\def\gessim{\ \lower4pt\hbox{$\buildrel{\displaystyle >}
		\over\sim$}\ }

\def\la{\langle}
\def\ra{\rangle}

\def\E{{\cal E}}

\def\ch{{\mbox{\rm ch}\hspace{0.4mm}}}
\def\sh{{\mbox{\rm sh}}}

\def\Av{{\mbox{\rm{Av}}}}
\def\exp{{\mbox{\rm exp}}}
\def\log{{\mbox{\rm log}}}

\def\Bla{\Big{\langle}}
\def\Bra{\Big{\rangle}}

\newcommand{\e}{\mathbb{E}}

\newtheorem{lemma}{\bf Lemma}

\newtheorem{theorem}{\bf Theorem}

\newtheorem{corollary}{\bf Corollary}

\newtheorem{proposition}{\bf Proposition}

\newenvironment{Proof of lemma}{\noindent{\bf Proof of Lemma}}{\hfill$\Box$\newline}
\newenvironment{Proof of theorem}{\noindent{\bf Proof of Theorem}}{\hfill{\footnotesize${\square}$}\newline}
\newenvironment{Proof of theorems}{\noindent{\bf Proof of Theorems}}{\hfill$\Box$\newline}
\newenvironment{Proof of proposition}{\noindent{\bf Proof of Proposition}}{\hfill$\Box$\newline}
\newenvironment{Proof of propositions}{\noindent{\bf Proof of Propositions}}{\hfill$\Box$\newline}
\newenvironment{Proof of exercise}{\noindent{\it Proof of Exercise:}}{\hfill$\Box$}
\begin{document}

\title{Thouless-Anderson-Palmer equations for the Ghatak-Sherrington mean field spin glass model}

\author{Antonio Auffinger \thanks{Department of Mathematics, Northwestern University, tuca@northwestern.edu, research partially supported by NSF Grant CAREER DMS-1653552, Simons Foundation/SFARI (597491-RWC), and  NSF Grant 1764421.} \\
\small{Northwestern University}
	\and Cathy Xi Chen 
	\thanks{Department of Mathematics, xchen56@math.northwestern.edu}
	\\ \small{Northwestern University}
}

\maketitle


\begin{abstract}
    We derive the Thouless-Anderson-Palmer (TAP) equations for the Ghatak and Sherrington model \cite{ghatak_crystal_1977}. Our derivation, based on the cavity method, holds  at high temperature and at all values of the crystal field. It confirms the prediction of \cite{yokota_first_order_1992}.
\end{abstract}


    \section{Introduction and main results}

The Hamiltonian of the Ghatak and Sherrington (GS) spin-glass model is defined as  the random function  
\begin{equation}
H_N(\boldsymbol \sigma) = \frac{\beta}{\sqrt{N}}\sum_{1\le i<j\le N}g_{ij} \sigma_i \sigma_j + D\sum_{i=1}^N \sigma^2_{i} + h\sum_{i=1}^N \sigma_{i},
\end{equation}
    where $S \geq 1$ is a fixed integer, and $\boldsymbol \sigma = \left(\sigma_1, \ldots, \sigma_N\right)\in \Sigma_N=\{0,\pm 1, \ldots, \pm S \}^N$. The parameters $\beta \ge 0$, $D \in \R$, $h \in \R$ represent the inverse temperature, crystal field and external field respectively, and $g_{ij}$ are i.i.d. standard Gaussian random variables for $1\le i<j\le N$. 
 This model was introduced by Ghatak and Sherrington in \cite{ghatak_crystal_1977} as a generalization of the classical   
Sherrington-Kirkpatrick (SK) model \cite{sherrington_solvable_1975}. It is supposed to model an induced spin glass and an anisotropic extension of the SK model \cite{ghatak_crystal_1977}.

As in the SK model, the study of thermodynamic quantities of the GS model has required significant efforts by many physicists and mathematicians. In particular, it has been predicted the existence of multiple phase transitions as the temperature decreases to zero,   including a second replica symmetric phase at low temperature, a phenomena indicative of inverse freezing \cite{Leuzzi}. This is in sharp contrast with the SK model (and the $p$-spin). We refer the reader to \cite{ghatak_crystal_1977, lage_stability_1982, mottishaw_stability_1985, yokota_first_order_1992, costa_first-order_1994, katsuki_katayama_ghatak-sherrington_1999 , da_costa_zero_temperature_2000}  and the references therein for a brief history and importance of the GS model in the physics community. In the mathematics literature, the most notable progress was Panchenko's result establishing an explicit formula for the limiting free energy \cite{panchenko_free_2005}.

In this paper, we study the behavior of the thermal average of the magnetization  
\[
m = (m_{1}, \ldots, m_{N}) = (\langle \sigma_{1} \rangle, \ldots, \langle \sigma_{N}\rangle )
\]
and its second moment 
\[
p = (p_{1}, \ldots, p_{N})=(\langle \sigma_{1}^{2} \rangle, \ldots, \langle \sigma_{N}^{2}\rangle ),
\]
where for a function $f$ on $\Sigma_N$, we denote $\la f \ra$ the average under the Gibbs measure $G_N$, defined as 
 $$G_N(\{\boldsymbol \sigma\})=\frac{\exp(H_N(\boldsymbol \sigma))}{Z_N},$$
    with 
    $$Z_N= \sum_{\boldsymbol \sigma}\exp(H_N(\boldsymbol \sigma)).$$
It has been predicted (in the case $S=1, h=0$ \cite{mottishaw_stability_1985, yokota_first_order_1992}) that these pairs of random variables satisfy at high temperature (in a sense that will be made precise later)  a system of coupled self consistent equations given by
\begin{equation}\label{eq:TAPphys}
m_{i} \approx  \frac{2\sinh(\beta \xi_{i})}{\exp(\beta \Delta_{i}) + 2 \cosh(\beta \xi_{i})} \quad \quad
p_{i}\approx  \frac{2\cosh(\beta \xi_{i})}{\exp(\beta \Delta_{i}) + 2 \cosh(\beta \xi_{i})}
\end{equation}
with 
\[
\xi_{i}= \frac{1}{\sqrt{N}}\sum_{i}g_{ij}m_{j} - \frac{\beta}{N} m_{i}\sum_{j} g_{ij}^{2}(p_{j}-m_{j}^{2}),
\]
and
\[
\Delta_{i}=-D-\frac{\beta}{2N} \sum_{j}g_{ij}^{2}(p_{j}-m_{j}^{2}).
\]


Equations \eqref{eq:TAPphys} are the analogue of the well-studied TAP equations 
\begin{equation}\label{eqTAPSK}
m_{i} \approx \tanh\left( h + \sum_{k\neq i} g_{ik}m_{k}-\beta^{2}(1-q)m_{i}\right)
\end{equation}
in the SK model \cite{thouless_solution_1977}. 
  
In the mathematics community, there have been several approaches to rigorously understand the TAP equations. First, in the SK model,  Talagrand \cite{talagrand_mean_2011} and Chatterjee \cite{chatterjee_spin_2010}  established \eqref{eqTAPSK}  at high temperature. At low temperature, a version of \eqref{eqTAPSK} where one decomposes the Gibbs measure into ``pure states'' was established by Auffinger-Jagannath \cite{AA_AJ_2019}. A very fruitful approach to TAP was introduced by Bolthausen through an iteration scheme that shares some connections to message passing algorithms \cite{bolthausen_iterative_2014}. Bolthausen's iteration was recently shown to indeed approximate the magnetizations by Chen-Tang\cite{chen_tang_2020}. A  dynamical method to derive \eqref{eqTAPSK} was also very recently proposed by  Adhikari-Brennecke-von Soosten-Yau \cite{adhikari_Yau_2021}. The TAP equations were also viewed as critical point solutions of the TAP functional and studied in \cite{AB, ABC,Subag}.
  
Different than the SK case or the mixed $p$-spin, the TAP equations for the GS model depend on two set of parameters. This creates a few roadblocks to understand its validity. For instance, in the physics community, there is still a debate of what should be the correct analogue of the de Almeida-Thouless line and for which set of parameters $(\beta, D, h)$ one should expect \eqref{eq:TAPphys} to be true. The main goal of this paper is  to derive a rigorous interpretation of $\eqref{eq:TAPphys}$ at high temperature for all values of the crystal field; we follow Talagrand's approach.  As far as we know, this is the first rigorous study of the TAP equations for this model and the first example where these equations manifest as a coupled system.

We will now state our results. Let $X$ be a standard Gaussian random variable. Given $(\beta, D, h)$ as above consider the system of equations in $\mathbb R^{2}$ given by  
 \begin{align} 
        p &= \e \left[\frac{\sum_{\gamma=1}^{S} \gamma^2 \cdot 2 \ch\left[\gamma \left( \sqrt{q}\beta X+h\right)\right]\exp\left(\gamma^2 \left[ D+\frac{\beta^2}{2}(p-q)\right]\right)}{1+ \sum_{\gamma=1}^{S} 2\ch\left[ \gamma \left(\sqrt{q}\beta X+h \right)\right]\exp\left(\gamma^2 \left[ D+\frac{\beta^2}{2}(p-q)\right]\right)}\right], \label{p}  \\
        q &= \e \left[\frac{ \sum_{\gamma=1}^{S} \gamma \cdot 2\sh\left[\gamma \left( \sqrt{q}\beta X+h\right)\right]\exp\left(\gamma^2 \left[ D+\frac{\beta^2}{2}(p-q)\right]\right)}{1+ \sum_{\gamma=1}^{S}  2\ch\left[\gamma \left( \sqrt{q}\beta X+h\right)\right]\exp\left(\gamma^2 \left[ D+\frac{\beta^2}{2}(p-q)\right]\right)}\right]^2. 
        \label{q}
    \end{align}  

This system is the analogue of the fixed-point equation $ q = \mathbb E \tanh^{2}(\beta\sqrt{q}X +h)$ that appears in the SK model. Our first result shows that for $\beta$ small, this system of equations has a unique solution.

    \begin{proposition} \label{pq eqn}
        There exists a $\tilde \beta > 0$ such that for all $0\leq \beta < \tilde{\beta},$ $h \geq 0$, and $D \in \R,$ the system of equations \eqref{p} and \eqref{q} has a unique solution.
    \end{proposition}

Assume that $\beta < \tilde \beta$ and the pair $(p,q)$ is the unique solutions of   \eqref{p} and \eqref{q}.  Our main result describes the validity of the TAP equations in the $L^{2}$ sense as follows.

    \begin{theorem}\label{TAP}
        There exists some $K, \hat{\beta} > 0$, such that for all $0 \leq \beta < \hat \beta$, $h\geq 0$ and $D\in \R$, we have for all $N\geq 1$
        \begin{align}
            \e \left[ \la \sigma_N \ra - \frac{ \sum_{\gamma=1}^{S}  \gamma \cdot 2 \sh\left[\gamma \left( \beta \xi_N + h\right)\right]\exp\left(\gamma^2\Delta \right)}{1 + \sum_{\gamma=1}^{S}  2 \ch \left[\gamma \left( \beta \xi_N + h\right)\right]\exp\left(\gamma^2\Delta \right)}\right]^2 &\le \frac{K}{\sqrt{N}}, \label{TAP1}\\
            \e \left[ \la \sigma_N^2 \ra - \frac{\sum_{\gamma=1}^{S}  \gamma^2 \cdot 2 \ch\left[\gamma \left( \beta \xi_N + h\right)\right]\exp\left(\gamma^2\Delta \right)}{1 + \sum_{\gamma=1}^{S} 2 \ch\left[\gamma \left( \beta \xi_N + h\right)\right]\exp\left(\gamma^2\Delta \right)}\right]^2 &\le \frac{K}{\sqrt{N}}, \label{TAP2}
        \end{align}
        where 
        \begin{align*}
            \xi_N &= \sum_{i \le N-1} g_{iN} \la \sigma_i \ra -\beta (p-q) \la \sigma_N \ra  , \, \text{and} \, \, \, \Delta = D + \frac{1}{2}\beta^2(p-q).
        \end{align*}
    \end{theorem}
  

    The proof of Theorem \ref{TAP} follows the cavity approach as in Section $1.6$ and $1.7$ of Talagrand's book \cite{talagrand_mean_2011}. The main difference between the SK model and the GS model is that we now need to control the self-overlap, and relate it to the solutions $(p,q)$. This requires new estimates and a careful analysis of the fixed point equation. The rest of the paper is organized as follows. In the next section, we show concentration of the overlap and self-overlap, the main tool to prove Theorem \ref{TAP}. In Section 3, we provide the proof of Theorem \ref{TAP}. The proof of Proposition \ref{pq eqn} is left to the last section.

\subsection{Acknowledgments}
Both authors would like to thank Wei-Kuo Chen for several suggestions on a previous version of this work, including a simplification of the proof of Proposition 1. They also would like to thank Si Tang for early discussions and help with computer simulations.

    \section{Concentration of overlaps}

    We denote the overlap between configurations $\boldsymbol \sigma^1$ and $\boldsymbol \sigma^2$, and the self-overlap of $\boldsymbol \sigma$ respectively by
    $$R_{1,2}=\frac{1}{N}\sum_{i \leq N}\sigma_i^1\sigma_i^2 \quad \text{and} \quad R_{1,1}=\frac{1}{N}\sum_{i \leq N}(\sigma_i)^2.$$
    In this section, we use the cavity method to show concentration of overlaps $R_{1,2}$ and $R_{1,1}$. We assume from now on that $\beta < \tilde \beta$ and $(p,q)$ are the solutions given in Proposition \ref{pq eqn}.
    
    \begin{proposition}\label{overlap}
        There exists a $\hat \beta > 0$, such that for all $\beta < \hat \beta $, we have:
        \begin{align} 
            \e  \left \la \left(R_{1,2}-q \right)^2 \right \ra  &\le \frac{16S^2}{N}, \notag \\
            \e \left \la \left( R_{1,1}-p  \right)^2 \right \ra &\le \frac{16S^4}{N}.\notag
        \end{align}
    \end{proposition}

    We start with some notations and preliminary results needed to prove Proposition \ref{overlap}. For $\boldsymbol \sigma = (\sigma_1, \ldots, \sigma_N)$, $\boldsymbol \rho = (\sigma_1, \ldots, \sigma_{N-1})$, we write 
    \begin{align}
        H_N(\boldsymbol \sigma) &= \frac{\beta}{\sqrt{N}}\sum_{1\le i<j\le N}g_{ij} \sigma_i \sigma_j + D\sum_{i=1}^N \sigma^2_{i} + h\sum_{i=1}^N \sigma_{i} \notag\\ 
        &=  H_{N-1}(\boldsymbol \rho) +  \sigma_N \cdot \frac{\beta}{\sqrt{N}}\sum_{ i< N}g_{iN} \sigma_i  + D\sigma_N^2   + h \sigma_N,\notag
    \end{align}
    where, with a slight abuse of notation,
    \begin{align} 
        H_{N-1}(\boldsymbol \rho)= \frac{\beta}{\sqrt{N}}\sum_{1\le i<j\le N-1}g_{ij} \sigma_i \sigma_j +D\sum_{i \le N-1}\sigma^2_{i}+h\sum_{i \le N-1}\sigma_{i}. \label{N-1 Hamil}
    \end{align}
    With the notation above, we have the following identity. Its proof is identical to the proof of Proposition 1.6.1 in \cite{talagrand_mean_2011}.

    \begin{proposition} \label{Expectation}
        Given a function $f$ on $\Sigma_N$, it holds that 
        \begin{align}
            \la f(\boldsymbol \sigma) \ra = \frac{\Bla \Av\left(f\left(\boldsymbol \sigma \right)\exp\left( \sigma_N \cdot \frac{\beta}{\sqrt{N}}\sum_{ i< N}g_{iN} \sigma_i  + D \sigma_N^2  + h \sigma_N    \right)\right)\Bra_{-}}{\Bla \Av\left(\exp\left( \sigma_N \cdot \frac{\beta}{\sqrt{N}}\sum_{ i< N}g_{iN} \sigma_i  +  D\sigma_N^2 + h \sigma_N \right)\right)\Bra_{-}}, \notag
        \end{align}
        where $\Av$ means average over $\sigma_N= 0,\pm 1, \ldots, \pm S$, and $\la \cdot \ra_{-}$ is the average under the Gibbs measure with respect to the Hamiltonian $H_{N-1}.$
    \end{proposition}
Now consider the interpolated Hamiltonian:
    \begin{align} 
        H_t(\boldsymbol \sigma) = H_{N-1}(\boldsymbol \rho) + \sigma_N \left[ \sqrt{t}\frac{\beta}{\sqrt{N}}\sum_{ i< N}g_{iN} \sigma_i + \sqrt{1-t} \beta z \sqrt{q} \right] +(1-t) \cdot \frac{\beta^2}{2}(p-q) \sigma_N^2 &+  D \sigma_N^2 \notag\\ 
        &\phantom{=}\ +  h \sigma_N, \notag
    \end{align}
    wherer $z$ is a standard Gaussian random variable independent of $g_{ij}$.
    We denote the overlap of the first $N-1$ coordinates by
    $$R_{l,l'}^- = \frac{1}{N}\sum_{i < N}\sigma_i^l \sigma_i^{l'}.$$
    To simplify the notation, let $\epsilon_l = \sigma_N^l$, and write
    \begin{align}
        R_{l,l'}= R_{l,l'}^- + \frac{\epsilon_l \epsilon _{l'}}{N}. \label{R-}
    \end{align}
\begin{lemma} \label{pq value}
        We have: 
        $$\e \la \sigma_N^2 \ra_0 =p \quad \text{and} \quad \e \la \sigma_N \ra_0^2 =q,$$
        where $p$ and $q$ satisfy the equations \eqref{p} and \eqref{q}, and $\la \cdot \ra_0$ is the the average under the Gibbs measure with respect to the interpolated Hamiltonian $H_{t}(\boldsymbol \sigma)$ at $t=0$.
    \end{lemma}

    \begin{proof}
        Note that 
        $$H_0(\boldsymbol \sigma) = H_{N-1}(\boldsymbol \rho) + \sigma_N  \cdot \beta z \sqrt{q} + \frac{\beta^2}{2}(p-q) \sigma_N^2 + D \sigma_N^2 + h \sigma_N. $$
        Applying Proposition \ref{Expectation} with $H_0$, we get:
        \begin{align*}
            \e \la \sigma_N^2 \ra_0 &= \e  \frac{\Bla \Av\left( \sigma_N^2  \exp\left( \sigma_N  \cdot \beta z \sqrt{q} + \frac{\beta^2}{2}(p-q) \sigma_N^2 + D \sigma_N^2 +  h \sigma_N \right)\right)\Bra_{-}}{\Bla \Av\left(\exp\left( \sigma_N  \cdot \beta z \sqrt{q} + \frac{\beta^2}{2}(p-q) \sigma_N^2 + D \sigma_N^2 +  h \sigma_N \right)\right)\Bra_{-}}\\
            &=\e \left[\frac{ \sum_{\gamma=1}^{S} \gamma^2 \cdot 2 \ch\left[\gamma \left( \sqrt{q}\beta X+h\right)\right]\exp\left(\gamma^2 \left[ D+\frac{\beta^2}{2}(p-q)\right]\right)}{1+\sum_{\gamma=1}^{S} 2\ch\left[\gamma \left( \sqrt{q}\beta X+h\right)\right]\exp\left(\gamma^2 \left[ D+\frac{\beta^2}{2}(p-q)\right]\right)}\right] = p.
        \end{align*}
        Similarly, 
        \begin{align*}
            \la \sigma_N \ra_0 &=  \frac{\Bla \Av\left( \sigma_N  \exp\left( \sigma_N  \cdot \beta z \sqrt{q} + \frac{\beta^2}{2}(p-q) \sigma_N^2 + D \sigma_N^2 +  h \sigma_N\right)\right)\Bra_{-}}{\Bla \Av\left(\exp\left( \sigma_N  \cdot \beta z \sqrt{q} + \frac{\beta^2}{2}(p-q) \sigma_N^2 + D \sigma_N^2 + h \sigma_N \right)\right)\Bra_{-}}\\
            &= \left[\frac{ \sum_{\gamma=1}^{S} \gamma \cdot 2 \sh\left[\gamma \left( \sqrt{q}\beta X+h\right)\right]\exp\left(\gamma^2 \left[ D+\frac{\beta^2}{2}(p-q)\right]\right)}{1+\sum_{\gamma=1}^{S} 2\ch\left[\gamma \left( \sqrt{q}\beta X+h\right)\right]\exp\left(\gamma^2 \left[ D+\frac{\beta^2}{2}(p-q)\right]\right)}\right],\\
            \e  \la \sigma_N \ra_0^2 &= \left[\frac{ \sum_{\gamma=1}^{S} \gamma \cdot 2 \sh\left[\gamma \left( \sqrt{q}\beta X+h\right)\right]\exp\left(\gamma^2 \left[ D+\frac{\beta^2}{2}(p-q)\right]\right)}{1+\sum_{\gamma=1}^{S} 2\ch\left[\gamma \left( \sqrt{q}\beta X+h\right)\right]\exp\left(\gamma^2 \left[ D+\frac{\beta^2}{2}(p-q)\right]\right)}\right]^2 = q.
 \end{align*}\end{proof}
Let $$u_{\boldsymbol \sigma} = \frac{\beta}{\sqrt{N}}\sum_{ i< N}g_{iN} \sigma_i\sigma_N, \, \,  v_{\boldsymbol \sigma} = \beta z \sqrt{q}\sigma_N,\,\, y_{\boldsymbol \sigma} = \frac{\beta^2}{2}(p-q) \sigma_N^2, \, \,\text{and} \, \, \omega_{\boldsymbol \sigma} = \exp\left(H_{N-1}(\boldsymbol \rho)+ D \sigma_N^2 + h \sigma_N\right). $$
    Then the interpolated Hamiltonian can be written as 
    $$H_t({\boldsymbol \sigma})=\sqrt{t}u_{\boldsymbol \sigma} + \sqrt{1-t}v_{\boldsymbol \sigma} + (1-t) y_{\boldsymbol \sigma} + \log(\omega_{\boldsymbol \sigma}).$$
   Let $\la \cdot \ra_t$ be an average for the corresponding Gibbs measure. We write 
   $$\nu_t(f)=\e \la f \ra_t,\,\, \text{and} \, \, \nu'_t(f)= \frac{d}{dt}(\nu_t(f)).$$
   Then $\e  \left \la \left(R_{1,2}-q \right)^2 \right \ra$ and $ \e \left \la \left( R_{1,1}-p  \right)^2 \right \ra$ in Proposition \ref{overlap} are equal to $\nu_1\left(\left(R_{1,2}-q\right)^2\right)$ and $ \nu_1\left(\left(R_{1,1}-p\right)^2\right)$. Set 
   $$U(\boldsymbol \sigma^l, \boldsymbol \sigma^{l'}) = \frac{1}{2}\left( \e u_{\boldsymbol \sigma^l}u_{\boldsymbol \sigma^{l'}} - \e v_{\boldsymbol \sigma^l}v_{\boldsymbol \sigma^{l'}} \right),\,\, \text{and} \, \, V(\boldsymbol \sigma^l) = -y_{\boldsymbol \sigma^l}. $$
   Then we have
   $$ \e u_{\boldsymbol \sigma^l}u_{\boldsymbol \sigma^{l'}} = \epsilon_l \epsilon_{l'} \cdot \beta^2 R_{l,l'}^- , \quad \e v_{\boldsymbol \sigma^l}v_{\boldsymbol \sigma^{l'}} =  \epsilon_l \epsilon_{l'} \cdot \beta^2q,$$
   and
   \begin{align}
        U(\boldsymbol \sigma^l, \boldsymbol \sigma^{l'}) &=  \epsilon_l \epsilon_{l'} \cdot \frac{\beta^2}{2}(R_{l,l'}^- -q), \label{U}\\
       V(\boldsymbol \sigma^l)  &= -\frac{\beta^2}{2}(p-q) \epsilon_l^2. \label{V}  
   \end{align}

    \begin{lemma}\label{v'} 
        If $f$ is a function on $(\Sigma_N)^n$, then 
        \begin{align}
            \nu_t'\left(f\left(\boldsymbol \sigma^1, \ldots, \boldsymbol \sigma^n\right)\right) &= \sum_{1 \le l, l' \le n} \nu_t \left(U\left(\boldsymbol \sigma^l, \boldsymbol \sigma^{l'}\right) f\right) - 2n \sum_{l \le n} \nu_t \left(U\left(\boldsymbol \sigma^l, \boldsymbol \sigma^{n+1}\right) f\right) \notag\\ 
             & \phantom{=}\ - n  \nu_t \left(U\left(\boldsymbol \sigma^{n+1}, \boldsymbol \sigma^{n+1}\right) f\right) + n(n+1) \nu_t \left(U\left(\boldsymbol \sigma^{n+1}, \boldsymbol \sigma^{n+2}\right) f\right) \notag\\
             & \phantom{=}\ + \sum_{l \le n} \nu_t \left(V\left(\boldsymbol \sigma^l\right) f\right) - n\nu_t \left(V\left(\boldsymbol \sigma^{n+1}\right) f\right) .  \label{v'(f)}
        \end{align}
    \end{lemma}

    \begin{proof}
        The proof is similar to  Lemma 1.4.2 in \cite{talagrand_mean_2011} with the correct derivatives.
    \end{proof}

    If we apply Lemma \ref{v'} to $H_t(\boldsymbol{ \sigma})$ and corresponding $U(\boldsymbol \sigma^l, \boldsymbol \sigma^{l'})$ and $V(\boldsymbol \sigma^l )$, we get the following.

    \begin{lemma}
        Let $f$ be a function on $(\Sigma_N)^n$, then for $0 < t < 1$, we have
        \begin{align}
            \nu'_t(f) &= \beta^2 \Bigl[ \sum_{1 \le l <l' \le n} \nu_t \left(\epsilon_{l}\epsilon_{l'}\left(R^-_{l,l'}-q\right)f\right) - n \sum_{l \le n} \nu_t \left(\epsilon_{l}\epsilon_{n+1}\left(R^-_{l,n+1}-q\right) f\right) \notag\\ 
            &\phantom{===}\ + \frac{n(n+1)}{2} \nu_t \left(\epsilon_{n+1}\epsilon_{n+2}\left(R^-_{n+1,n+2}-q\right) f\right) + \frac{1}{2} \sum_{l \le n} \nu_t \left(\epsilon_l^2 \left(R^-_{l,l} - p\right)f \right) \notag\\
            &\phantom{===}\ - \frac{n}{2}  \nu_t \left(\epsilon_{n+1}^2 \left(R^-_{n+1,n+1} - p\right) f\right),\label{nu'(f)} 
        \end{align}
        and also 
        \begin{align}
            \nu'_t(f) = A_1 + A_2 - B ,  \label{v'R} 
        \end{align}
        where 
        \begin{align}
            A_1 &:= \beta^2 \Bigl[ \sum_{1 \le l <l' \le n} \nu_t \left(\epsilon_{l}\epsilon_{l'}\left(R_{l,l'}-q\right)f\right) - n \sum_{l \le n} \nu_t \left(\epsilon_{l}\epsilon_{n+1}\left(R_{l,n+1}-q\right) f\right) \notag\\ 
            & \phantom{===}\ + \frac{n(n+1)}{2} \nu_t \left(\epsilon_{n+1}\epsilon_{n+2}\left(R_{n+1,n+2}-q\right) f\right) \Bigr], \label{A1}\\ 
            A_2 &:= \beta^2 \Bigl[\frac{1}{2} \sum_{l \le n} \nu_t \left(\epsilon_l^2 \left(R_{l,l} - p\right)f \right) - \frac{n}{2}  \nu_t \left(\epsilon_{n+1}^2 \left(R_{n+1,n+1} - p\right) f\right) \Bigr], \label{A2}\\ 
            B &:= \frac{\beta^2}{N}  \Bigl[ \sum_{1 \le l <l' \le n} \nu_t \left(\epsilon_{l}^2\epsilon_{l'}^2 f\right) - n \sum_{l \le n} \nu_t \left(\epsilon_{l}^2\epsilon_{n+1}^2 f\right)+ \frac{n(n+1)}{2} \nu_t \left(\epsilon_{n+1}^2\epsilon_{n+2}^2f\right) \notag\\ 
            & \phantom{===}\ + \frac{1}{2} \sum_{l \le n} \nu_t \left(\epsilon_l^4 f \right)  - \frac{n}{2}  \nu_t \left(\epsilon_{n+1}^4  f\right) \Bigr]. \label{B} 
        \end{align}
    \end{lemma}

    \begin{proof}
       The result follows by replacing equation \eqref{U}  and \eqref{V} in \eqref{v'(f)} and by some straightforward algebra.
    \end{proof}

    \begin{lemma}\label{ineq_nu_t and 1}
        For a function $f \geq 0$ on $(\Sigma_N)^n$, we have 
        \begin{align}
            \nu_t(f) \le \exp\left(6 n^2\beta^2 S^4 \right) \nu_1(f). \notag
        \end{align}
    \end{lemma}
   
    \begin{proof}
        Note that $R^-_{l.l'} \le S^2$ and $p,q \in [0,S^2]$. So $\lvert R^-_{l.l'} - q \rvert \le 2S^2 $ for $l \neq l'$, $\lvert R^-_{l.l} - p \rvert \le 2S^2 $, and $(p-q)^2 \le S^2$. Thus, by \eqref{nu'(f)}, we have
        \begin{align}
           \lvert \nu'_t(f)\rvert  &\le \left( \frac{n(n-1)}{2} + n^2 + \frac{n(n+1)}{2} + \frac{n}{2} + \frac{n}{2}\right) \beta^2 \cdot 2S^4\nu_t(f) \notag\\
            &= (2n^2 + n) \beta^2 \cdot 2S^4\nu_t(f) \le 6n^2\beta^2S^4\nu_t(f). \label{|nu'(f)|} 
        \end{align}
        Thus,
        $$ \left \lvert \frac{\nu'_t(f)}{\nu_t(f)} \right \rvert \le 6n^2\beta ^2 S^4.  $$
        Let
        $$ g(1-t):= \log\left[v_t(f)\right]\quad \text{for} \quad t\in [0,1].$$
        Then
        $$\left \lvert g'(1-t) \right \rvert = \left \lvert \frac{\nu'_t(f)}{\nu_t(f)}\right \rvert \le 6n^2\beta ^2S^4.$$
        Therefore,
        $$ g(1-t) = g(0) + \int_0^{1-t} g'(s)ds \le g(0) + \int_0^{1-t} \lvert g'(s) \rvert ds \le g(0) + (1-t) 6n^2\beta ^2 S^4,$$
        i.e.
        $$\log[\nu_t(f)] \le \log[\nu_1(f)]+ (1-t) 6n^2\beta ^2 S^4. $$
        Hence,
        $$\nu_t(f) \le \exp\left[(1-t)6n^2\beta^2S^4\right] \nu_1(f) \le \exp \left(6n^2\beta^2S^4\right) \nu_1(f), \, \, \text{as desired.} $$
    \end{proof}

    \begin{lemma} \label{difference nu_1&0}
       Given a function $ f$ on $(\Sigma_N)^n$, and $\tau_1, \tau_2 >0$ with $\frac{1}{\tau_1} + \frac{1}{\tau_2} = 1$, we have 
       \begin{align}
          \left \lvert \nu_1(f) - \nu_0(f) \right \rvert \le \exp \left[6n^2\beta^2S^4\right] &\biggl(  2n^2 \beta^2S^2\left[\nu_1(|f|^{\tau_1})\right]^{\frac{1}{\tau_1}}\left[\nu_1(\lvert R_{12} - q \rvert^{\tau_2})\right]^{\frac{1}{\tau_2}} \notag\\
         & \phantom{===}\ + n^2 \beta^2 S^4 \left(2 + \frac{3}{N} \right)\nu_1(|f|) \biggr), \label{difference nu_1&0_R12}
       \end{align}
       and
       \begin{align}
        \left \lvert \nu_1(f) - \nu_0(f) \right \rvert \le \exp \left[6n^2\beta^2S^4\right] &\biggl(  2n^2 \beta^2S^2\left[\nu_1(|f|^{\tau_1})\right]^{\frac{1}{\tau_1}}\left[\nu_1(\lvert R_{11} - p \rvert^{\tau_2})\right]^{\frac{1}{\tau_2}} \notag\\
       & \phantom{===}\ + n^2 \beta^2 S^4 \left(4 + \frac{3}{N} \right)\nu_1(|f|) \biggr). \label{difference nu_1&0_R11}
     \end{align}
    \end{lemma}

    \begin{proof}
        We show \eqref{difference nu_1&0_R12} first. Note that 
        $$\left \lvert \nu_1(f) - \nu_0(f) \right \rvert = \left \lvert \int^1_0 \nu'_t(f) dt \right \rvert \le \sup_{0 \le t \le 1} \left \lvert \nu'_t(f) \right \rvert. $$
        Also know that $\lvert \epsilon_{l}\epsilon_{l'} \rvert \le S^2$ and $(p-q)^2 \le S^2$.
        Now apply H\"{o}lder's inequalitiy, we will have for $\frac{1}{\tau_1} + \frac{1}{\tau_2} = 1,$ and $l \neq l'$,
        \begin{align}
           \left \lvert  \nu_t \left(\epsilon_{l}\epsilon_{l'}\left(R_{l,l'}-q\right)f\right) \right \rvert \le S^2\nu_t\left(|f||R_{l,l'} - q|\right) \le S^2\left[\nu_t(|f|^{\tau_1})\right]^{\frac{1}{\tau_1}}\left[\nu_t(\lvert R_{12} - q \rvert^{\tau_2})\right]^{\frac{1}{\tau_2}}.  \notag
        \end{align}
        Also we have
        $$\nu_t\left(\epsilon^2_{l}\epsilon^2_{l'}f\right) \le S^4\nu_t (|f|).$$
        According to equations \eqref{A1} to \eqref{B}, 
        \begin{align}
            |A_1| &\le \left( \frac{n(n-1)}{2} + n^2 + \frac{n(n+1)}{2} \right) \beta^2S^2 \left[\nu_t(|f|^{\tau_1})\right]^{\frac{1}{\tau_1}}\left[\nu_t(\lvert R_{12} - q \rvert^{\tau_2})\right]^{\frac{1}{\tau_2}} \notag\\
            &= 2n^2\beta^2 S^2 \left[\nu_t(|f|^{\tau_1})\right]^{\frac{1}{\tau_1}}\left[\nu_t (\lvert R_{12} - q \rvert^{\tau_2})\right]^{\frac{1}{\tau_2}}, \notag\\
            |A_2| &\le \beta^2 \left[ \frac{n}{2}\cdot 2S^4\nu_t(|f|) + \frac{n}{2}\cdot 2S^4\nu_t(|f|)\right] =2\beta^2n S^4\nu_t(|f|) , \notag\\
            |B| &\le \frac{\beta^2}{N} \left( \frac{n(n-1)}{2} + n^2 + \frac{n(n+1)}{2} + \frac{n}{2} + \frac{n}{2}\right) S^4 \nu_t(|f|) \notag\\
            &= \frac{\beta^2}{N} \left(2n^2 + n\right)S^4 \nu_t(|f|) \le \frac{3\beta^2n^2S^4}{N}\nu_t(|f|). \notag
        \end{align}
        Hence based on equation \eqref{v'R}, we have
        \begin{align}
            |\nu'_t(f)| &\le |A_1| + |A_2| + |B| \notag\\
            &\le 2n^2\beta^2 S^2\left[\nu_t(|f|^{\tau_1})\right]^{\frac{1}{\tau_1}}\left[\nu_t(\lvert R_{12} - q \rvert^{\tau_2})\right]^{\frac{1}{\tau_2}} + n^2\beta^2 S^4\left(2 + \frac{3}{N}   \right) \nu_t(|f|) \notag.
        \end{align}
        Now by Lemma \ref{ineq_nu_t and 1}, we get
        \begin{align}
            |\nu'_t(f)|  \le \exp \left[6n^2\beta^2S^4\right] &\biggl(  2n^2 \beta^2S^2\left[\nu_1(|f|^{\tau_1})\right]^{\frac{1}{\tau_1}}\left[\nu_1(\lvert R_{12} - q \rvert^{\tau_2})\right]^{\frac{1}{\tau_2}} \notag\\
         & \phantom{===}\ + n^2 \beta^2 S^4 \left(2 + \frac{3}{N} \right)\nu_1(|f|) \biggr). \notag
        \end{align}
        Therefore, 
        \begin{align}
            \left \lvert \nu_1(f) - \nu_0(f) \right \rvert \le \sup_{0 \le t \le 1} \left \lvert \nu'_t(f) \right \rvert \le \exp \left[6n^2\beta^2S^4\right] &\biggl(  2n^2 \beta^2S^2 \left[\nu_1(|f|^{\tau_1})\right]^{\frac{1}{\tau_1}}\left[\nu_1(\lvert R_{12} - q \rvert^{\tau_2})\right]^{\frac{1}{\tau_2}} \notag\\
            & \phantom{==}\ + n^2 \beta^2S^4 \left(2 + \frac{3}{N} \right)\nu_1(|f|) \biggr). \notag
        \end{align}
        Now we show the inequality \eqref{difference nu_1&0_R11}. Note that we have for $\frac{1}{\tau_1} + \frac{1}{\tau_2} = 1,$ and $l =l'$,
        \begin{align}
            \left \lvert  \nu_t \left(\epsilon_{l}\epsilon_{l'}\left(R_{l,l}-p\right)f\right) \right \rvert \le S^2\left[\nu_t(|f|^{\tau_1})\right]^{\frac{1}{\tau_1}}\left[\nu_t(\lvert R_{11} - p \rvert^{\tau_2})\right]^{\frac{1}{\tau_2}}.  \notag
         \end{align}
         Considering different upper bounds for $|A_1|$ and $|A_2|$, we obtain
         \begin{align*}
            |A_1| &\le \left( \frac{n(n-1)}{2} + n^2 + \frac{n(n+1)}{2} \right) \beta^2S^2 \cdot 2S^2 \nu_t \left(|f| \right) = 4n^2\beta^2 S^4  \nu_t \left(|f| \right), \notag\\
            |A_2| &\le \left( \frac{2}{n} + \frac{2}{n} \right)\beta^2S^2 \left[\nu_t(|f|^{\tau_1})\right]^{\frac{1}{\tau_1}}\left[\nu_t(\lvert R_{11} - p \rvert^{\tau_2})\right]^{\frac{1}{\tau_2}}  \le n^2\beta^2S^2 \left[\nu_t(|f|^{\tau_1})\right]^{\frac{1}{\tau_1}}\left[\nu_t(\lvert R_{11} - p \rvert^{\tau_2})\right]^{\frac{1}{\tau_2}}. \notag
         \end{align*}
         Following the same method as above, we show \eqref{difference nu_1&0_R11} as desired. 
    \end{proof}

   We now prove Proposition \ref{overlap}.
    
    \begin{proof}[\bf Proof of Proposition \ref{overlap}]
        We will show concentration of $R_{1,2}$ first. Recall that $\epsilon_l = \sigma_N^l$. Using symmetry among replicas, we can write
        $$\nu_1\left(\left(R_{1,2}-q\right)^2\right) = \frac{1}{N} \sum_{i \le N} \nu_1\left[\left(\sigma_i^1 \sigma_i^2 -  q\right)\left(R_{1,2} -  q\right)\right] = \nu_1(f),$$
        where
        $$f := \left(\epsilon_1\epsilon_2 - q\right) \left(R_{1,2} - q\right).$$
        By \eqref{R-},
        $$f = \left(\epsilon_1 \epsilon_2 - q\right) \left(\frac{\epsilon_1 \epsilon _{2}}{N} + R^-_{1,2} - q\right) = \frac{1}{N}\left[\left(\epsilon_1\epsilon_2\right)^2 - \epsilon_1\epsilon_2 q\right] +  \left(\epsilon_1\epsilon_2 - q\right)\left(R^-_{1,2} -  q\right).$$
        Lemma \ref{pq value} implies
        $$\nu_0 \left[\left(\epsilon_1\epsilon_2 - q\right)\left(R^-_{1,2} -  q\right)\right] =  \nu_0 \left(\epsilon_1\epsilon_2 - q\right) \nu_0\left(R^-_{1,2} -  q\right) = \left[ \nu_0(\epsilon_1)\nu_0(\epsilon_2) - q\right]\nu_0\left(R^-_{1,2} -  q\right) = 0,$$
        and hence 
       \begin{align}
            \nu_0(f)=  \frac{1}{N} \nu_0\left[\left(\epsilon_1\epsilon_2\right)^2 - \epsilon_1\epsilon_2 q\right] =  \frac{1}{N} \left[\nu_0\left(\epsilon^2_1 \right)\nu_0\left(\epsilon^2_2 \right) - \nu_0\left(\epsilon_1^2 \right)q\right] = \frac{1}{N}(p^2 - q^2). \label{nu0f}
       \end{align}
        Using $\left \lvert \epsilon_1\epsilon_2 - q \right \rvert \le 2S^2$, we have 
        $$\left \lvert f \right \rvert = \left \lvert  \left(\epsilon_1\epsilon_2 - q\right) \left(R_{1,2} - q\right)  \right \rvert \le 2 S^2\left \lvert R_{1,2} - q \right \rvert.$$
        Now apply Lemma \ref{difference nu_1&0} with $\tau_1 = \tau_2 = 2,$ and $n=2$ to get
        \begin{align}
            \left \lvert \nu_1(f) - \nu_0(f) \right \rvert \le \exp \left[24\beta^2S^4\right] &\biggl(  16\beta^2S^4 \nu_1(\lvert R_{12} - q \rvert^2)\notag\\
            & \phantom{===}\ + 4 \beta^2 S^4 \left(2 + \frac{3}{N} \right)\nu_1\left(\left \lvert R_{1,2}-q\right \rvert^2 \right) \biggr). \label{v1-v0}
        \end{align}
        Therefore, combining \eqref{v1-v0} with equation \eqref{nu0f}, we obtain 
        \begin{align}
            \nu_1\left(\left \lvert R_{1,2}-q\right \rvert^2 \right) \le \frac{1}{N}(p^2 - q^2) + 12 \beta^2 S^4\left(2 + \frac{1}{N} \right)  \exp \left[24\beta^2S^4\right] \nu_1\left(\left \lvert R_{1,2}-q\right \rvert^2 \right). \notag
        \end{align}
        Note that $p,q \in [0,S^2]$ and $p \ge q$, then $0 \le p^2 - q^2 \le S^2$.
        Choose $\beta_0$ such that 
        $$ 12\beta_0^2 S^4\left( 2 + \frac{1}{N} \right)  \exp \left[24\beta_0^2S^4\right] \le \frac{15}{16},$$
        then we have 
        $$  \nu_1\left(\left \lvert R_{1,2}-q\right \rvert^2 \right) \le \frac{S^2}{N} + \frac{15}{16}   \nu_1\left(\left \lvert R_{1,2}-q\right \rvert^2 \right),$$
        and hence
        \begin{align}
            \nu_1\left(\left \lvert R_{1,2}-q\right \rvert^2 \right) \le \frac{16S^2}{N}. \label{ineq R12}
        \end{align}
        We use a similar method to show concentration of $R_{1,1}$. We can write
        $$\nu_1\left(\left(R_{1,1}-p\right)^2\right) = \frac{1}{N} \sum_{i \le N} \nu_1\left[\left((\sigma_i^1)^2  -  p\right)\left(R_{1,1} -  p\right)\right] = \nu_1(g),$$
        where
        $$\kappa := \left(\epsilon_1^2  - p\right) \left(R_{1,1} - p\right).$$
        It follows that 
        $$\kappa = \left(\epsilon_1^2  - p \right) \left(\frac{\epsilon_1^2} {N} + R^-_{1,1} - p \right) = \frac{1}{N}\left[\left(\epsilon_1^2 \right)^2 - \epsilon_1^2 p \right] +  \left(\epsilon_1^2  - p\right)\left(R^-_{1,1} -  p\right),$$
        and by Lemma \ref{pq value} 
        $$\nu_0 \left[\left(\epsilon_1^2- p\right)\left(R^-_{1,1} -  p\right)\right] =  \nu_0 \left(\epsilon_1^2 - p\right) \nu_0\left(R^-_{1,1} -  p\right) = \left[ \nu_0(\epsilon_1^2) - p\right]\nu_0\left(R^-_{1,1} -  p\right) = 0.$$
        Hence, 
       \begin{align}
            \nu_0(\kappa)=  \frac{1}{N} \nu_0\left[\left(\epsilon_1^2\right)^2 - \epsilon_1^2 p\right] \le  \frac{1}{N} \left[\nu_0\left(S^4 \right) - \nu_0\left(\epsilon_1^2 \right)p\right] \le \frac{1}{N}(S^4 - p^2) \le \frac{S^4}{N}. \label{nu0g}
       \end{align}
        By definition of $\kappa$, $\nu_1\left(\left(R_{1,1}-p\right)^2\right) = \nu_1(\kappa)$. Also note that $R_{1,1} , \epsilon_1^2, p \in [0,S^2]$. We have 
        $$\left \lvert \epsilon_1^2 - p \right \rvert \le S^2 \quad \text{and} \quad |R_{1,1}-p| \le S^2,$$
        thus, 
        $$\left \lvert \kappa \right \rvert = \left \lvert  \left(\epsilon_1^2 - p\right) \left(R_{1,1} - p\right)  \right \rvert \le S^4 |R_{1,1}-p|.$$
       By Lemma \ref{difference nu_1&0} with $\tau_1 = \tau_2 = 2,$ and $n=2$, we get
        \begin{align}
            \left \lvert \nu_1(\kappa) - \nu_0(\kappa) \right \rvert \le \exp \left[24\beta^2S^4\right] &\biggl(  4 \beta^2 S^4\left[\nu_1(\lvert R_{11} - p \rvert^2)\right]\notag\\
            & \phantom{===}\ + 4 \beta^2 S^4 \left(4 + \frac{3}{N} \right)\nu_1\left(\left \lvert R_{1,1}-p\right \rvert^2 \right) \biggr). \label{house}
        \end{align}
        Therefore, combining \eqref{house} with \eqref{nu0g}, we obtain
        \begin{align}
            \nu_1\left(\left \lvert R_{1,1}-p\right \rvert^2 \right) &\le \frac{S^4}{N} + 4\beta^2 S^4 \left(5+ \frac{3}{N} \right)  \exp \left[24\beta^2S^4\right] \nu_1\left(\left \lvert R_{1,1}- p\right \rvert^2 \right). \notag
        \end{align}
        Choosing $\beta_1$ such that 
        \begin{align*}
                4\beta_1^2 S^4\left(5 + \frac{3}{N} \right)  \exp \left[24\beta_1^2S^4\right] \le \frac{15}{16},
        \end{align*}       
        we have 
        $$  \nu_1\left(\left \lvert R_{1,1}-p\right \rvert^2 \right) \le \frac{S^4}{N} + \frac{15}{16}   \nu_1\left(\left \lvert R_{1,1}-p\right \rvert^2 \right),$$
        i.e.
        $$ \nu_1\left(\left \lvert R_{1,1}-p\right \rvert^2 \right) \le \frac{16S^4}{N}.$$
        Now take $\hat{\beta} =  \min (\beta_0, \beta_1)$. Then for all $\beta < \hat{\beta},$ we have:
        $$ \nu_1\left(\left \lvert R_{1,2}-q\right \rvert^2 \right) \le \frac{16S^2}{N} \quad \text{and}\quad  \nu_1\left(\left \lvert R_{1,1}-p\right \rvert^2 \right) \le \frac{16S^4}{N}.$$
    \end{proof}

    \section{TAP equations for the Ghatak-Sherrington model}

    In this section, we prove Theorem \ref{TAP}. Set $\beta_-$ to be
    $$\frac{\beta_-}{\sqrt{N-1}} = \frac{\beta}{\sqrt{N}}.$$ 
   Note that for some positive constant $K$, $$\lvert \beta - \beta_- \rvert \le  \frac{K}{N}.$$ Let $p_- = p_-(N-1)$ and $q_- = q_-(N-1)$ be so that 
    \begin{align}
        p_- &= \e \left[\frac{\sum_{\gamma=1}^{S} \gamma^2 \cdot 2 \ch\left[\gamma \left( \sqrt{q_-}\beta_- X+h\right)\right]\exp\left(\gamma^2 \left[ D+\frac{\beta_-^2}{2}(p_--q_-)\right]\right)}{1+ \sum_{\gamma=1}^{S} 2\ch\left[ \gamma \left(\sqrt{q_-}\beta_- X+h \right)\right]\exp\left(\gamma^2 \left[ D+\frac{\beta_-^2}{2}(p_--q_-)\right]\right)}\right], \notag\\
        q_- &= \e \left[\frac{ \sum_{\gamma=1}^{S} \gamma \cdot 2\sh\left[\gamma \left( \sqrt{q_-}\beta_- X+h\right)\right]\exp\left(\gamma^2 \left[ D+\frac{\beta_-^2}{2}(p_--q_-)\right]\right)}{1+ \sum_{\gamma=1}^{S}  2\ch\left[\gamma \left( \sqrt{q_-}\beta_- X+h\right)\right]\exp\left(\gamma^2 \left[ D+\frac{\beta_-^2}{2}(p_--q_-)\right]\right)}\right]^2. \notag
    \end{align}
    We have the following lemma.

    \begin{lemma} \label{lem pp- qq-}
        There exists a $K,\hat{\beta} >0 $ such that for all $\beta < \hat{\beta}$, $D,h \in \R$, we have 
        \begin{align}
            |p - p_-| &\le \frac{K}{N}, \label{pp-}\\
            |q - q_-| &\le \frac{K}{N}. \label{qq-}
        \end{align}
    \end{lemma}

    \begin{proof}
        We will show \eqref{pp-} first, and \eqref{qq-} follows similarly. Let
        $$\phi (X) = \frac{\sum_{\gamma=1}^{S} \gamma^2 \cdot 2 \ch\left[\gamma \left( \sqrt{q}\beta X+h\right)\right]\exp\left(\gamma^2 \left[ D+\frac{\beta^2}{2}(p-q)\right]\right)}{1+ \sum_{\gamma=1}^{S} 2\ch\left[ \gamma \left(\sqrt{q}\beta X+h \right)\right]\exp\left(\gamma^2 \left[ D+\frac{\beta^2}{2}(p-q)\right]\right)},$$
        and 
        $$\psi (X) = \frac{ \sum_{\gamma=1}^{S} \gamma \cdot 2\sh\left[\gamma \left( \sqrt{q}\beta X+h\right)\right]\exp\left(\gamma^2 \left[ D+\frac{\beta^2}{2}(p-q)\right]\right)}{1+ \sum_{\gamma=1}^{S}  2\ch\left[\gamma \left( \sqrt{q}\beta X+h\right)\right]\exp\left(\gamma^2 \left[ D+\frac{\beta^2}{2}(p-q)\right]\right)}. $$
        Define
        $$G\left(\beta, p, q\right) = \e \left[\phi(X)\right],$$
        and
        $$F\left(\beta, p, q \right) = \e \left[\psi(X)\right]^2 .$$
        Also define $p(\beta), q(\beta)$ by
        \begin{align}
            p(\beta) = G\left(\beta, p(\beta), q(\beta)\right), \, \,  \text{and} \, \, \,   q(\beta) = F\left(\beta, p(\beta), q(\beta)\right). \notag
        \end{align}
        We have 
        \begin{align}
            q'(\beta) = \frac{\frac{\partial F}{\partial \beta} + \frac{\partial F}{\partial p} \cdot p'(\beta)}{1 - \frac{\partial F}{\partial q}}, \, \,  \text{and} \, \, \, p'(\beta) = \frac{\frac{\partial G}{\partial \beta} + \frac{\partial G}{\partial q} \cdot q'(\beta)}{1 - \frac{\partial G}{\partial p}}. \notag
        \end{align}
        Plug the equation of $q'(\beta)$ into $p'(\beta)$, then we obtain
        \begin{align}
            p'(\beta) = \frac{\frac{\partial G}{\partial \beta} \left(1 - \frac{\partial F}{\partial q}\right)  + \frac{\partial G}{\partial q} \cdot \frac{\partial F}{\partial \beta}}{\left(1 - \frac{\partial G}{\partial p}\right) \left(1 - \frac{\partial F}{\partial q}\right) - \frac{\partial G}{\partial q} \cdot \frac{\partial F}{\partial p}}. \notag
        \end{align}
        Now we calculate the partial derivatives of $G\left(\beta, p, q \right)$ and $F\left(\beta, p, q \right)$. To simplify the notation, we introduce the following functions.
        Define
  \begin{align}\label{def:fL6} f(X) &= \frac{1}{1+ \sum_{\gamma=1}^{S}  2\ch\left[\gamma \left( \sqrt{q}\beta X+h\right)\right]\exp\left(\gamma^2 \left[ D+\frac{\beta^2}{2}(p-q)\right]\right)},\\
        \kappa_{ch}(X) &= 2 \ch\left[\gamma \left( \sqrt{q}\beta X+h\right)\right]\exp\left(\gamma^2 \left[ D+\frac{\beta^2}{2}(p-q)\right]\right), \label{def:fL7}\\
               \kappa_{sh} (X)&= 2\sh\left[\gamma \left( \sqrt{q}\beta X+h\right)\right]\exp\left(\gamma^2 \left[ D+\frac{\beta^2}{2}(p-q)\right]\right). \label{def:fL8}
       \end{align}
        Then 
        $$\phi (X) =  \sum_{\gamma=1}^S \gamma^2 \cdot \kappa_{ch}(X)f(X), \,\, \text{and} \,\, \,\psi (X) =  \sum_{\gamma=1}^S \gamma \cdot \kappa_{sh}(X)f(X). $$
        Define 
        $$\theta(X) = \sum_{\gamma=1}^S \gamma^4 \cdot \kappa_{ch}(X)f(X), \,\, \text{and} \, \,\,\eta (X) =  \sum_{\gamma=1}^S \gamma^3 \cdot \kappa_{sh}(X)f(X).$$
        Since $\beta < \hat{\beta}$, and $D,h \in \R$, it's clear that $\phi(X), \psi(X), f(X), \theta(X)$, and $\eta(X)$ are bounded functions. By some straight-forward algebra, it is easy to see that the functions $\phi'(X), \psi'(X), f'(X), \theta'(X)$, and $\eta'(X)$ are also bounded. 
        Hence,
        \begin{align*}
            \frac{\partial G}{\partial \beta} &= \e \left[ \sqrt{q} \left(\eta'(X) - \psi'(X)\phi(X) - \psi(X)\phi'(X)\right) + \beta(p-q)\left(\theta(X) - \phi^2(X)\right) \right], \\  
            \frac{\partial G}{\partial p} &= \e \left[\frac{\beta^2}{2} \left( \theta(X)  - \phi^2(X)\right) \right],\\
            \frac{\partial G}{\partial q} &= \e \left[ \frac{\beta}{2\sqrt{q}} \left(\eta'(X) - \psi'(X)\phi(X) - \psi(X)\phi'(X)\right) + \frac{\beta^2}{2}\left(\phi^2(X) - \theta(X) \right)\right],\\
            \frac{\partial F}{\partial \beta} &= 2 \e \left[\sqrt{q} \left[\psi'(X)\phi(X) + \psi(X)\phi'(X)\right] +\beta(p-q)\psi(X)\left[\eta(X)- \phi(X)\right] - 3\psi^2(X) \psi'(X) \right],\\
            \frac{\partial F}{\partial p} &= \e \left[ \beta^2 \psi(X) \left(\eta(X) - \psi(X)\phi(X)\right) \right],\\
            \frac{\partial F}{\partial q} &= \e \left[ \frac{\beta}{\sqrt{q}} \left(\psi'(X)\phi(X) + \psi(X)\phi'(X) - 3\psi^2(X)\psi'(X)\right) + \beta^2 \left(\psi^2(X)\phi(X) - \psi(X)\eta(X)\right)\right].
        \end{align*}
       It follows that all the partial derivatives of $G\left(\beta, p, q \right)$ and $F\left(\beta, p, q \right)$ are bounded, hence so is $p'(\beta)$.
        Note that for some positive number $K$, 
        $$\lvert \beta - \beta_- \rvert \le  \frac{K}{N}.$$
        By the mean value theorem, 
        $$|p - p_-| \le \frac{K}{N}.$$
        Similarly, $|q - q_-| \le \frac{K}{N}$ is satisfied.   
    \end{proof}
      
    Before the proof of Theorem \ref{TAP},  we state a result which is the analogue of Theorem 1.7.11 in \cite{talagrand_mean_2011}.  Consider independent standard Gaussian random variables $y_{i}$ and $\xi$, which are independent
    of the randomness of $\la \cdot \ra$, and denote $\e_{\xi}$ the expectation with respect to the random variable $\xi$ only. 
    
    \begin{theorem}(\cite{talagrand_mean_2011}, Theorem 1.7.11)\label{appro of local field}
        Assume $\beta < \hat{\beta}$,  and $D, h \in \R$. Let $U$ be an infinitely differentiable function on $\R$ with derivatives given by $U^{(l)}$. Assume for all $l$ and $b$, the $l^{th}$ derivative of $U$ satisfies
        \begin{align*}
            \e |U^{(l)}(z)|^b < \infty
        \end{align*}
        where $z$ is a Gaussian random variable. Then, using the notation $\dot{\sigma}_i = \sigma_i - \la \sigma_i \ra$, we have for $k=1,2$
        \begin{align*}
            \e \left(\left \la U\left(\frac{1}{\sqrt{N}}\sum_{i \le N} y_{i} \dot{\sigma}_i\right) \right \ra - \e_{\xi}U \left(\xi \sqrt{p-q}\right) \right)^{2k} \le \frac{K}{N},
        \end{align*}
        where $p$ and $q$ satisfy the equations \eqref{p} and \eqref{q} respectively, and the constant $K$ depends on $U,\beta$, but not on $N$.
    \end{theorem}

    \begin{proof}
        The proof of Theorem \ref{appro of local field} is similar to Talagrand's proof of the Theorem 1.7.11 \cite{talagrand_mean_2011} except for the following differences. We write the case $k=1$ first.
        Let $\dot{S_l} =  \frac{1}{\sqrt{N}}\sum_{i \le N}y_i\dot{\sigma}_i^l$ and $\e_0$ denote the expectation with respect to $y_{i}$ and $\xi_l$ only. 
        Set 
        $$T_{l,l} = \e_0 (\dot{S_l}^2) -  \e_0 (\xi_l \sqrt{p-q})^2 = \frac{1}{N}\sum_{i \le N}(\dot{\sigma}_i^l)^2 - (p-q),$$
        and for $l \ne l'$, let
        $$T_{l,l'} = \e_0 (\dot{S_l}\dot{S_{l'}}) -  \e_0 \left[\xi_l \xi_{l'}(p-q)\right] = \frac{1}{N}\sum_{i \le N}(\dot{\sigma}_i^l)(\dot{\sigma}_i^{l'}).$$
        We claim that there exists a positive number $K$ such that 
        $$\e \la T^2_{l,l'} \ra \le \frac{K}{N}.$$ We explain the case $l=l'$ first.
        Since  
        $$(\dot{\sigma}_i^l)^2 = \left(\sigma_i^l - \la \sigma_i \ra \right)^2 = (\sigma_i^l)^2 - 2\sigma_i^l \la \sigma_i^l \ra + \la \sigma_i \ra^2,$$
        it follows that
        \begin{align}
            T_{l,l} &= \frac{1}{N}\sum_{i \le N}(\dot{\sigma}_i^l)^2 - (p-q) =  \frac{1}{N}\sum_{i \le N} \left(  (\sigma_i^l)^2 - 2\sigma_i^l \la \sigma_i^l \ra + \la \sigma_i \ra^2 \right) - (p-q) \notag \\
            &= \frac{1}{N}\sum_{i \le N}\left[(\sigma_i^l)^2 - p\right] + \frac{1}{N}\sum_{i \le N} \left[ \left(\la \sigma_i \ra^2 - q\right) + 2 \left(q - \sigma_i^l \la \sigma_i^l \ra \right)\right]. \notag
        \end{align}
        We control the first and second term of this sum separately. 
        By Proposition \ref{overlap}, 
        \begin{align}
              \e \left \la \left[\frac{1}{N}\sum_{i \le N}(\sigma_i^l)^2 - p\right]^2 \right \ra = \e \left \la (R_{1,1}- p)^2\right \ra \le \frac{16S^4}{{N}} \label{sum1 term}
        \end{align} 
        For the second term of the sum, we use the fact that for any $A$ and $B$, we have the inequality $(A + B)^2 \le 2(A^2 + B^2)$. Apply Jensen's inequality and Proposition \ref{overlap}, we have
       \begin{align}
           &\e \left \la \left[ \frac{1}{N}\sum_{i \le N}  \left(\la \sigma_i \ra^2 - q\right) + 2 \left(q - \sigma_i^l \la \sigma_i^l \ra \right)\right]^2 \right \ra \notag \\
           &\phantom{=======}\ \le 2\left(  \e \left \la \left[ \frac{1}{N}\sum_{i \le N}  \left(\la \sigma_i \ra^2 - q\right) \right]^2 \right \ra +  \e \left \la \left[  \frac{1}{N}\sum_{i \le N}  2 \left(q - \sigma_i^l \la \sigma_i^l \ra \right)\right]^2 \right \ra \right) \notag \\
           &\phantom{=======}\ \le 2\left( \e \left \la (R_{1,2} - q)^2\right \ra  + 4 \e \left \la (q- R_{1.2})^2\right \ra \right) \notag \\
           &\phantom{=======}\ \le 2S^2\left( \frac{16}{N} + \frac{64}{N} \right) = \frac{160S^2}{N}. \label{sum2 term}
       \end{align}
       Therefore, combining \eqref{sum1 term} and \eqref{sum2 term}, we obtain that for some positive number $K$
       \begin{align}
        \e \la T_{l,l}^2 \ra &\le 2\left(\e \left \la \left[\frac{1}{N}\sum_{i \le N}(\sigma_i^l)^2 - p\right]^2 \right \ra  + \e \left \la \left[ \frac{1}{N}\sum_{i \le N}  \left(\la \sigma_i \ra^2 - q\right) + 2 \left(q - \sigma_i^l \la \sigma_i^l \ra \right)\right]^2 \right \ra \right) \notag\\ 
        & \phantom{==========================================}\ \le \frac{K}{N}. \notag
       \end{align}
       As for the case $l \neq l'$
       \begin{align}
           T_{l,l'} = \frac{1}{N}\sum_{i \le N}(\dot{\sigma}_i^l)(\dot{\sigma}_i^{l'}) =  \frac{1}{N}\sum_{i \le N} \left( \sigma_i^{l} \sigma_i^{l'}- \sigma_i^l \la \sigma_i \ra -\sigma_i^{l'} \la \sigma_i \ra + \la \sigma_i \ra^2\right), \notag
       \end{align}
       thus, 
       $$\e \la T_{l,l'} \ra = \e \frac{1}{N}\sum_{i \le N} \left( \la \sigma_i^{l} \sigma_i^{l'} \ra - \la \sigma_i^l \ra \la \sigma_i \ra -\la \sigma_i^{l'} \ra \la \sigma_i \ra + \la \sigma_i \ra^2\right) = 0 \le \frac{K}{N}.$$
       Hence, for all $l$ and $l'$, and some positive number $K$, we have
       $$\e \la T^2_{l,l'} \ra \le \frac{K}{N}.$$
       This estimate replaces inequality (1.201) with $r=1$ in Talagrand's book \cite{talagrand_mean_2011}. For the case $k=2$ we proceed similarly and use the bounds 
\begin{align*}
\mathbb E \la (R_{12} - q)^{4}\ra &\le 4S^{2} \mathbb E \la (R_{12} - q)^{2} \ra \leq \frac{64 S^{4}}{N},\\
\mathbb E \la (R_{11} - p)^{4}\ra &\le 4S^{2} \mathbb E \la (R_{11} - p)^{2} \ra \leq \frac{64 S^{6}}{N}.
\end{align*}
    \end{proof}
    Now we state two corollaries of the Theorem \ref{appro of local field}, which are equivalent of  Talagrand's Corollary 1.7.13 and 1.7.15 \cite{talagrand_mean_2011}. 

    \begin{corollary}
        There exists a $K, \hat{\beta} >0 $ such that for all $\beta < \hat{\beta}$, $D,h \in \R$, and $\epsilon \in \Sigma$ we have
        \begin{align}
            \e \left( \left \la \exp \frac{\epsilon \beta}{\sqrt{N}} \sum_{i \le N} y_i \sigma_i \right \ra - \exp \left[ \frac{\epsilon^2 \beta^2}{2}(p-q) \right] \exp \frac{\epsilon \beta}{\sqrt{N}} \sum_{i \le N} y_i \la \sigma_i \ra \right)^2 \le \frac{K}{\sqrt{N}}, \label{cor ine 1}
         \end{align}
        and
        \begin{align}
            \e  &\left( \left \la\frac{1}{\sqrt{N}} \sum_{i \le N} y_i \dot{\sigma}_i  \exp \frac{\epsilon \beta}{\sqrt{N}} \sum_{i \le N} y_i \sigma_i \right \ra - \epsilon\beta(p-q)\exp \left[\frac{\epsilon^2 \beta^2}{2}(p-q)\right] \exp \frac{\epsilon \beta}{\sqrt{N}} \sum_{i \le N} y_i \la \sigma_i \ra\right)^2 \notag\\ 
            & \phantom{==============================================}\ \le \frac{K}{\sqrt{N}} \label{cor ine 2}
        \end{align}
        where K does not depend on N.
    \end{corollary}

    \begin{proof}
The proof is identical to the proof of Corollary 1.7.13 in \cite{talagrand_mean_2011}. 
    \end{proof}
For the rest of the section, we will use the following lemma.
\begin{lemma} \label{ineq}
        If $\left \vert \frac{A'}{B'}\right \vert \le B$ and $B \ge 1$, we have 
        $$ \left \lvert \frac{A'}{B'} - \frac{A}{B}\right \rvert \le |A - A'| + |B - B'|. $$
    \end{lemma}


    \begin{corollary}
        Let
        \begin{align}
            \E = \exp \left( \frac{\epsilon \beta}{\sqrt{N}} \sum_{i \le N} y_i \sigma_i + \epsilon^2 D + \epsilon h\right). \label{Epsilon} 
        \end{align} 
        Recall that $\Av$ denotes average over $\epsilon \in \Sigma$. There exists a constant $K>0$ and $\hat{\beta }> 0$ such that for all $\beta < \hat{\beta}$, and  $D,h \in \R$, we have 
        \begin{align}
            \e \left(\frac{\left \la \Av \epsilon \E \right \ra }{\left \la  \Av \E \right \ra} - \frac{\sum_{\gamma=1}^{S} \gamma \cdot 2 \sh \left[ \gamma\left( \frac{\beta}{\sqrt{N}} \sum_{i \le N} y_i \la \sigma_i \ra + h\right)\right]\exp\left(\gamma^2 \left[ D+\frac{\beta^2}{2}(p-q)\right]\right)}{1+ \sum_{\gamma=1}^{S} 2\ch \left[ \gamma\left( \frac{\beta}{\sqrt{N}} \sum_{i \le N} y_i \la \sigma_i \ra + h\right)\right]\exp\left(\gamma^2 \left[ D+\frac{\beta^2}{2}(p-q)\right]\right)}  \right)^2 \le \frac{K}{\sqrt{N}}, \label{cor ine 3} 
        \end{align}
        \begin{align}
            \e \left(\frac{\left \la \Av \epsilon^2 \E \right \ra }{\left \la  \Av \E \right \ra} - \frac{\sum_{\gamma=1}^{S} \gamma^2 \cdot 2 \ch  \left[ \gamma\left( \frac{\beta}{\sqrt{N}} \sum_{i \le N} y_i \la \sigma_i \ra + h\right)\right] \exp\left(\gamma^2 \left[ D+\frac{\beta^2}{2}(p-q)\right]\right)}{1+\sum_{\gamma=1}^{S} 2\ch \left[ \gamma\left( \frac{\beta}{\sqrt{N}} \sum_{i \le N} y_i \la \sigma_i \ra + h\right)\right] \exp\left(\gamma^2 \left[ D+\frac{\beta^2}{2}(p-q)\right]\right)}  \right)^2 \le \frac{K}{\sqrt{N}}, \label{cor ine 5} 
        \end{align}
        \begin{align}
            \e \left( \frac{1}{\sqrt{N}} \sum_{i \le N} y_i \frac{\left \la \sigma_i \Av \E \right \ra }{\left \la  \Av \E \right \ra} - \beta(p -q)\frac{\left \la \Av \epsilon \E \right \ra }{\left \la  \Av \E \right \ra} - \frac{ 1}{\sqrt{N}} \sum_{i \le N} y_i \la \sigma_i \ra \right)^2 \le \frac{K}{\sqrt{N}}. \label{cor ine 4}
        \end{align}
    \end{corollary}

    \begin{proof}
        Define
        $$A(\epsilon) = \left \la \exp \frac{\epsilon \beta}{\sqrt{N}} \sum_{i \le N} y_i \sigma_i \right \ra - \exp \left[ \frac{\epsilon^2 \beta^2}{2}(p-q) \right] \exp \frac{\epsilon \beta}{\sqrt{N}} \sum_{i \le N} y_i \la \sigma_i \ra. $$
        Note that $A(0)=0$. Deducing from \eqref{cor ine 1}, for $\gamma= 1, \ldots, S,$ we have
        \begin{align}
            \e \left( \gamma A(\gamma) \exp (\gamma^2 D + \gamma h) - \gamma A(-\gamma) \exp (\gamma^2 D - \gamma h)\right)^2 \le \frac{K}{\sqrt{N}}, \notag
        \end{align}
        and
        \begin{align}
            \e \left(A(\gamma) \exp (\gamma^2 D + \gamma h) + A(-\gamma) \exp (\gamma^2 D - \gamma h) +  A(0) \right)^2 \le \frac{K}{\sqrt{N}}. \notag
        \end{align}
        Hence,
        \begin{align}
            \e \left( \sum^{S}_{\gamma=1} \left[\gamma A(\gamma) \exp (\gamma^2 D + \gamma h) - \gamma A(-\gamma) \exp (\gamma^2 D - \gamma h)\right]\right)^2 \le \frac{K}{\sqrt{N}}, \notag
        \end{align}    
        i.e.
        \begin{align}
            \e \left(\left \la \Av \epsilon \E \right \ra - \sum_{\gamma=1}^{S} \gamma \cdot 2 \sh \left[ \gamma\left( \frac{\beta}{\sqrt{N}} \sum_{i \le N} y_i \la \sigma_i \ra + h\right)\right]\exp\left(\gamma^2 \left[ D+\frac{\beta^2}{2}(p-q)\right]\right)  \right)^2 \le \frac{K}{\sqrt{N}}, \notag
        \end{align}
        \begin{align}
            \e \left(\left \la \Av \epsilon^2 \E \right \ra - \sum_{\gamma=1}^{S} \gamma^2 \cdot 2 \ch  \left[ \gamma\left( \frac{\beta}{\sqrt{N}} \sum_{i \le N} y_i \la \sigma_i \ra + h\right)\right] \exp\left(\gamma^2 \left[ D+\frac{\beta^2}{2}(p-q)\right] \right)  \right)^2 \le \frac{K}{\sqrt{N}}, \notag
        \end{align}
        and 
        \begin{align}
            \e \left(\left \la \Av \E \right \ra - \left[ 1+ \sum_{\gamma=1}^{S} 2\ch \left[ \gamma\left( \frac{\beta}{\sqrt{N}} \sum_{i \le N} y_i \la \sigma_i \ra + h\right)\right]\exp\left(\gamma^2 \left[ D+\frac{\beta^2}{2}(p-q)\right]\right)\right] \right)^2 \le \frac{K}{\sqrt{N}}. \label{cor ineq 6}
        \end{align}
        Equations \eqref{cor ine 3} and \eqref{cor ine 5} follow from  Lemma \ref{ineq}. Using the same method, we get from \eqref{cor ine 2}
        \begin{align}
            \e \left( \left \la \frac{1}{\sqrt{N}} \sum_{i \le N} y_i \dot{\sigma}_i  \Av \E \right \ra - \beta(p-q) \sum_{\gamma =1}^{S} \gamma \cdot \exp \left[\frac{ \gamma^2 \beta^2}{2}(p-q)\right] 2\sh \left[ \gamma\left( \frac{\beta}{\sqrt{N}} \sum_{i \le N} y_i \la \sigma_i \ra + h\right)\right] \right)^2 \notag \\
             \le \frac{K}{\sqrt{N}}.  \label{cor ineq 7}
        \end{align}
        Combining \eqref{cor ineq 7} with \eqref{cor ineq 6} and using Lemma \ref{ineq}, we obtain 
        \begin{align}
            \e &\Bigg(\frac{\left \la \frac{1}{\sqrt{N}} \sum_{i \le N} y_i \dot{\sigma}_i  \Av \E \right \ra }{\left \la  \Av \E \right \ra} \notag\\ 
            &\phantom{=}\ - \beta(p-q) \frac{\sum_{\gamma=1}^{S} \gamma \cdot 2 \sh \left[ \gamma\left( \frac{\beta}{\sqrt{N}} \sum_{i \le N} y_i \la \sigma_i \ra + h\right)\right]\exp\left(\gamma^2 \left[ D+\frac{\beta^2}{2}(p-q)\right]\right)}{1+ \sum_{\gamma=1}^{S} 2\ch \left[ \gamma\left( \frac{\beta}{\sqrt{N}} \sum_{i \le N} y_i \la \sigma_i \ra + h\right)\right]\exp\left(\gamma^2 \left[ D+\frac{\beta^2}{2}(p-q)\right]\right)}  \Bigg)^2 \le \frac{K}{\sqrt{N}}.  \label{cor ineq 8}
        \end{align}
        Note that 
        $$\frac{\left \la \frac{1}{\sqrt{N}} \sum_{i \le N} y_i \dot{\sigma}_i  \Av \E \right \ra }{\left \la  \Av \E \right \ra} = \frac{1}{\sqrt{N}} \sum_{i \le N} y_i \frac{\left \la \sigma_i \Av \E \right \ra }{\left \la  \Av \E \right \ra}  - \frac{ 1}{\sqrt{N}} \sum_{i \le N} y_i \la \sigma_i \ra.$$
        Combining \eqref{cor ineq 8} with \eqref{cor ine 3} proves \eqref{cor ine 4}.      
    \end{proof}

    Finally, we turn to the proof of Theorem \ref{TAP}.
    \begin{proof}[\bf Proof of Theorem \ref{TAP}] 
        First we will show \eqref{TAP1}. Recall that the Hamiltonian \eqref{N-1 Hamil} is the Hamiltonian of an $(N-1)$-spin system with parameter 
        $$\beta_- = \beta \sqrt{1-\frac{1}{N}} \le \beta ,$$
        $$ \E = \exp \left( \frac{\epsilon \beta_-}{\sqrt{N-1}} \sum_{i \le N-1} g_{iN} \sigma_i + \epsilon^2  D + \epsilon h \right) = \exp \left( \frac{\epsilon \beta}{\sqrt{N}} \sum_{i \le N-1} g_{iN} \sigma_i + \epsilon^2  D + \epsilon h \right). $$
        By Proposition \ref{Expectation}, we have 
        $$\la \sigma_N \ra =  \frac{\la \Av \epsilon \E \ra_-}{\la \Av \E \ra_-}.$$
        Next, applying \eqref{cor ine 3} to the $(N-1)$-spin system and the sequence $y_i=g_{iN}$, we obtain
        \begin{align}
            &\e \left(\la \sigma_N \ra - \frac{\sum_{\gamma=1}^{S} \gamma \cdot 2 \sh \left[ \gamma\left( \frac{\beta}{\sqrt{N}} \sum_{i \le N} g_{iN} \la \sigma_i \ra_- + h\right)\right]\exp\left(\gamma^2 \left[ D+\frac{\beta^2_-}{2}(p_- -q_- )\right]\right)}{1+ \sum_{\gamma=1}^{S} 2\ch \left[ \gamma\left( \frac{\beta}{\sqrt{N}} \sum_{i \le N} g_{iN} \la \sigma_i \ra_- + h\right)\right]\exp\left(\gamma^2 \left[ D+\frac{\beta^2_-}{2}(p_- -q_-)\right]\right)}  \right)^2 \notag \\ 
            &\phantom{=============================================}\le \frac{K}{\sqrt{N}}. \notag 
        \end{align}
        Now to show \eqref{TAP1}, it suffices to show that 
        \begin{align*}
            &\e \Bigg( \frac{\sum_{\gamma=1}^{S} \gamma \cdot 2 \sh \left[ \gamma\left( \frac{\beta}{\sqrt{N}} \sum_{i \le N} g_{iN} \la \sigma_i \ra + h\right)\right]\exp\left(\gamma^2 \left[ D+\frac{\beta^2}{2}(p-q)\right]\right)}{1+ \sum_{\gamma=1}^{S} 2\ch \left[ \gamma\left( \frac{\beta}{\sqrt{N}} \sum_{i \le N} g_{iN} \la \sigma_i \ra + h\right)\right]\exp\left(\gamma^2 \left[ D+\frac{\beta^2}{2}(p-q)\right]\right)}\\
            &\phantom{=}\ - \frac{\sum_{\gamma=1}^{S} \gamma \cdot 2 \sh \left[ \gamma\left( \frac{\beta}{\sqrt{N}} \sum_{i \le N} g_{iN} \la \sigma_i \ra_- + h\right)\right]\exp\left(\gamma^2 \left[ D+\frac{\beta^2_-}{2}(p_- -q_- )\right]\right)}{1+ \sum_{\gamma=1}^{S} 2\ch \left[ \gamma\left( \frac{\beta}{\sqrt{N}} \sum_{i \le N} g_{iN} \la \sigma_i \ra_- + h\right)\right]\exp\left(\gamma^2 \left[ D+\frac{\beta^2_-}{2}(p_- -q_-)\right]\right)}   \Bigg)^2 \le \frac{K}{\sqrt{N}}. 
        \end{align*}
        Let 
        $$f(x,y) = \frac{\sum_{\gamma=1}^{S} \gamma \cdot 2 \sh(\gamma y)e^{\gamma^2 x}}{ 1 + \sum_{\gamma=1}^{S} 2\ch(\gamma y)e^{\gamma^2 x}}.$$
        and let
        \begin{align*}
                x_1=  D+\frac{\beta^2}{2}(p - q), \quad 
                y_1 = \frac{\beta}{\sqrt{N}} \sum_{i \le N-1} g_{iN} \la \sigma_i \ra - \beta^2 (p-q) \la \sigma_N \ra + h,    
        \end{align*}
        and
        \begin{align*}  
                x_2 = D+\frac{\beta_-^2}{2}(p_- -q_-), \quad
                y_2 = \frac{ \beta}{\sqrt{N}} \sum_{i \le N-1} g_{iN} \la \sigma_i \ra_- + h .
        \end{align*}
        We claim
        $$\e \left[ f(x_1,y_1) - f(x_2,y_2) \right] ^2 \le \frac{K}{\sqrt{N}}.$$
        By taking the partial derivatives of $f(x,y)$, it is straight-forward to show that $f(x,y)$ is a Lipschitz function with respect to both $x$ and $y$. There exists a positive number $L$ such that 
        $$\lvert f(x_1,y) - f(x_2,y) \rvert \le L |x_1 - x_2|,$$
        and
        $$\lvert f(x,y_1) - f(x,y_2) \rvert \le L |y_1 - y_2|.$$
        Thus using the fact that for any $A$ and $B$, $(A + B)^2 \le 2(A^2 + B^2)$, we obtain 
        $$\e \left[ f(x_1,y_1) - f(x_2,y_2) \right]^2 \le 2 L^2 \left[ \e \left(x_1 - x_2\right)^2  + \e \left(y_1 - y_2\right)^2\right].$$
        By Lemma \ref{difference nu_1&0} and \ref{lem pp- qq-}, it follows that 
        \begin{align*}
            |x_1 - x_2| &= \left \lvert \frac{\beta^2}{2}(p - q) - \frac{\beta_-^2}{2}(p_- -q_-) \right \rvert\\
            & \le \left \lvert \frac{p - q}{2} \left(\beta^2 - \beta_-^2 \right) \right \rvert + \left \lvert \frac{\beta_-^2}{2} \left(p - p_- \right) \right \rvert + \left \lvert \frac{\beta_-^2}{2} \left(q - q_- \right) \right \rvert \le \frac{K}{\sqrt{N}},
        \end{align*}
        i.e.
        $$\e \left(x_1 - x_2\right)^2 \le \frac{K}{\sqrt{N}}.$$
        Now applying \eqref{cor ine 4} to the $(N-1)$-spin system, we get
        \begin{align*}
            \e \left( \frac{1}{\sqrt{N-1}} \sum_{i \le N-1} g_{iN} \left \la \sigma_i  \right \ra - \beta_-(p_- -q_-)\left \la \sigma_i  \right \ra - \frac{ 1}{\sqrt{N-1}} \sum_{i \le N-1} g_{iN} \la \sigma_i \ra_- \right)^2 \le \frac{K}{\sqrt{N}}.
        \end{align*}
        If we multiply both sides by $\beta_-^2$, using $ \left \lvert \beta^2 - \beta_-^2 \right \rvert \le \frac{K}{\sqrt{N}}$ and Lemma \ref{lem pp- qq-} again, we have\
        \begin{align*}
            \e \left( \frac{\beta }{\sqrt{N}} \sum_{i \le N-1} g_{iN} \left \la \sigma_i  \right \ra - \beta^2 (p - q )\left \la \sigma_i  \right \ra - \frac{ \beta}{\sqrt{N}} \sum_{i \le N-1} g_{iN} \la \sigma_i \ra_- \right)^2 \le \frac{K}{\sqrt{N}},
        \end{align*}
        i.e.
        $$\e \left(y_1 - y_2\right)^2 \le \frac{K}{\sqrt{N}}.$$
        Therefore, we have
        $$\e \left[ f(x_1,y_1) - f(x_2,y_2) \right]^2 \le 2 L^2 \left[ \e \left(x_1 - x_2\right)^2  + \e \left(y_1 - y_2\right)^2\right] \le \frac{K}{\sqrt{N}}.$$
        Similarly, we can show \eqref{TAP2} using the same method. 
    \end{proof}

    \section{Proof of Proposition \ref{pq eqn}} \label{sec 4}

 In this proof, we use the same notation as in the proof of Lemma \ref{lem pp- qq-}.

    \begin{proof}
        Recall that 
        $$\phi (X) = \frac{\sum_{\gamma=1}^{S} \gamma^2 \cdot 2 \ch\left[\gamma \left( \sqrt{q}\beta X+h\right)\right]\exp\left(\gamma^2 \left[ D+\frac{\beta^2}{2}(p-q)\right]\right)}{1+ \sum_{\gamma=1}^{S} 2\ch\left[ \gamma \left(\sqrt{q}\beta X+h \right)\right]\exp\left(\gamma^2 \left[ D+\frac{\beta^2}{2}(p-q)\right]\right)},$$
        and 
        $$\psi (X) = \frac{ \sum_{\gamma=1}^{S} \gamma \cdot 2\sh\left[\gamma \left( \sqrt{q}\beta X+h\right)\right]\exp\left(\gamma^2 \left[ D+\frac{\beta^2}{2}(p-q)\right]\right)}{1+ \sum_{\gamma=1}^{S}  2\ch\left[\gamma \left( \sqrt{q}\beta X+h\right)\right]\exp\left(\gamma^2 \left[ D+\frac{\beta^2}{2}(p-q)\right]\right)}. $$
        We define functions $G\left(\beta, p, q\right)= \e \left[\phi(X)\right]$ and $F\left(\beta, p, q \right) = \e \left[\psi(X)\right]^2$, and hence the equations \eqref{p} and \eqref{q} become 
        $$ p = G\left(\beta, p, q\right), \,\,\, \text{and} \,\,\,  q = F\left(\beta, p, q \right). $$
        Define a self-mapping $T: \left[0,S^2\right] \times \left[0,S^2\right] \rightarrow \left[0,S^2\right] \times \left[0,S^2\right]$ by  $$T(p,q):=\left(G\left(\beta, p, q\right),F\left(\beta, p, q \right) \right).$$
By the contraction mapping theorem, it suffices to show that there exists a $\tilde \beta > 0$ such that for all $0\leq \beta < \tilde{\beta},$ $h \geq 0$, and $D \in \R,$ $T$ is a contraction. 

        We have
        $$\phi (X) =  \sum_{\gamma=1}^S \gamma^2 \cdot \kappa_{ch}(X)f(X) \leq S^2, \,\,  \,\psi (X) =  \sum_{\gamma=1}^S \gamma \cdot \kappa_{sh}(X)f(X) \leq S, $$
        $$\theta(X) = \sum_{\gamma=1}^S \gamma^4 \cdot \kappa_{ch}(X)f(X) \leq S^4, \,\, \text{and} \, \,\,\eta (X) =  \sum_{\gamma=1}^S \gamma^3 \cdot \kappa_{sh}(X)f(X) \leq S^3,$$
where $f$, $\kappa_{sh}$ and $\kappa_{ch}$ are given by \eqref{def:fL6},  \eqref{def:fL7}, and  \eqref{def:fL8}. Calculating the derivatives of the above functions, we obtain 
        $$\phi' (X) = \sqrt{q}\beta\left(\eta(X) - \phi(X)\psi(X)\right) \leq 2\sqrt{q}\beta S^3,  \,\, \, \psi' (X) = \sqrt{q}\beta \left(\phi(X) - \psi^2(X)\right) \leq 2\sqrt{q}\beta S^2, $$
        and
        $$ \eta'(X) =  \sqrt{q}\beta \left(\theta(X) - \eta(X)\psi(X)\right) \leq 2 \sqrt{q}\beta S^4.$$
        Therefore, we have the following: 
        \begin{align*}
            \frac{\partial G}{\partial p} &= \e \left[\frac{\beta^2}{2} \left( \theta(X)  - \phi^2(X)\right) \right] \leq S^4\beta^2 := L_1,\\
            \frac{\partial G}{\partial q} &= \e \left[ \frac{\beta}{2\sqrt{q}} \left(\eta'(X) - \psi'(X)\phi(X) - \psi(X)\phi'(X)\right) + \frac{\beta^2}{2}\left(\phi^2(X) - \theta(X) \right)\right] \leq 4S^4\beta^2 :=L_2,\\
            \frac{\partial F}{\partial p} &= \e \left[ \beta^2 \psi(X) \left(\eta(X) - \psi(X)\phi(X)\right) \right] \leq 2S^4\beta^2 :=L_3,\\
            \frac{\partial F}{\partial q} &= \e \left[ \frac{\beta}{\sqrt{q}} \left(\psi'(X)\phi(X) + \psi(X)\phi'(X) - 3\psi^2(X)\psi'(X)\right) + \beta^2 \left(\psi^2(X)\phi(X) - \psi(X)\eta(X)\right)\right] \\
            &\phantom{==========================================}\ \leq 12S^4\beta^2 :=L_4.
        \end{align*}
        By Cauchy's inequality, we have 
        $$\lvert G\left(p_1,q_1\right) - G\left(p_2,q_2\right) \rvert \leq L_1 \lvert p_1 - q_1 \rvert + L_2 \lvert p_2 - q_2 \rvert \leq \sqrt{L_1^2 + L_2^2} \sqrt{\lvert p_1 - q_1 \rvert^2 + \lvert p_2 - q_2 \rvert^2},$$
        and similarly,
        $$\lvert F\left(p_1,q_1\right) - F\left(p_2,q_2\right) \rvert \leq L_3 \lvert p_1 - q_1 \rvert + L_4 \lvert p_2 - q_2 \rvert \leq \sqrt{L_3^2 + L_4^2} \sqrt{\lvert p_1 - q_1 \rvert^2 + \lvert p_2 - q_2 \rvert^2}.$$
        Hence,
        \begin{align*}
            \lvert T(p_1,q_1) - T(p_2, q_2)\rvert &= \sqrt{\left[ G\left(p_1,q_1\right) - G\left(p_2,q_2\right)\right]^2 + \left[ F\left(p_1,q_1\right) - F\left(p_2,q_2\right)\right]^2}\\
            &\leq \sqrt{\sum_{i=1}^4 L_i^2} \cdot \sqrt{\lvert p_1 - q_1 \rvert^2 + \lvert p_2 - q_2 \rvert^2} = \sqrt{165} S^4 \beta^2 \sqrt{\lvert p_1 - q_1 \rvert^2 + \lvert p_2 - q_2 \rvert^2}.
        \end{align*}
        To make this map a contraction map, we need $\sqrt{165} S^4 \beta^2 < 1$. Let $\tilde{\beta} = \frac{1}{\sqrt[4]{165}S^2}$. Thus, for all $0\leq \beta < \tilde{\beta},$ $h \geq 0$, and $D \in \R,$ $T$ is a contraction mapping.

    \end{proof}

\bibliographystyle{acm}

\bibliography{reference-1}

\begin{thebibliography}{10}

\bibitem{adhikari_Yau_2021}
{\sc Adhikari, A., Brennecke, C., von Soosten, P., and Yau, H.-T.}
\newblock Dynamical {Approach} to the {TAP} {Equations} for the
  {Sherrington}-{Kirkpatrick} {Model}.
\newblock arXiv: 2102.10178.

\bibitem{AB}
{\sc Auffinger, A., and Ben~Arous, G.}
\newblock Complexity of random smooth functions on the high-dimensional sphere.
\newblock {\em Ann. Probab. 41}, 6 (2013), 4214--4247.

\bibitem{ABC}
{\sc Auffinger, A., Ben~Arous, G., and \v{C}ern\'{y}, J.}
\newblock Random matrices and complexity of spin glasses.
\newblock {\em Comm. Pure Appl. Math. 66}, 2 (2013), 165--201.

\bibitem{AA_AJ_2019}
{\sc Auffinger, A., and Jagannath, A.}
\newblock Thouless-anderson-palmer equations for generic p-spin glasses.
\newblock {\em Annals of Probability 47}, 4 (2019), 2230--2256.

\bibitem{bolthausen_iterative_2014}
{\sc Bolthausen, E.}
\newblock An {Iterative} {Construction} of {Solutions} of the {TAP} {Equations}
  for the {Sherrington}-{Kirkpatrick} {Model}.
\newblock {\em Communications in Mathematical Physics 325}, 1 (Jan. 2014),
  333--366.

\bibitem{chatterjee_spin_2010}
{\sc Chatterjee, S.}
\newblock Spin glasses and {Stein}{\textquoteright}s method.
\newblock {\em Probability Theory and Related Fields 148}, 3-4 (Nov. 2010),
  567--600.

\bibitem{chen_tang_2020}
{\sc Chen, W.-K., and Tang, S.}
\newblock On convergence of {Bolthausen}'s {TAP} iteration to the local
  magnetization.
\newblock arXiv: 2011.00495.

\bibitem{da_costa_zero_temperature_2000}
{\sc Costa, F., and Ara\'ujo, J.}
\newblock Zero-temperature {TAP} equations for the {Ghatak}-{Sherrington}
  model.
\newblock {\em The European Physical Journal B 15}, 2 (May 2000), 313--316.

\bibitem{costa_first-order_1994}
{\sc Costa, F., Yokoi, C. S.~O., and Salinas, S. R.~A.}
\newblock First-order transition in a spin-glass model.
\newblock {\em Journal of Physics A: Mathematical and General 27}, 10 (May
  1994), 3365--3372.

\bibitem{ghatak_crystal_1977}
{\sc Ghatak, S.~K., and Sherrington, D.}
\newblock Crystal field effects in a general {S} {Ising} spin glass.
\newblock {\em Journal of Physics C: Solid State Physics 10}, 16 (Aug. 1977),
  3149--3156.

\bibitem{katsuki_katayama_ghatak-sherrington_1999}
{\sc Katayama, K., and Horiguchi, T.}
\newblock Ghatak-sherrington model with spin s.
\newblock {\em Journal of the Physical Society of Japan 68}, 12 (1999),
  3901--3910.

\bibitem{lage_stability_1982}
{\sc Lage, E. J.~S., and Almeida, J. R. L.~d.}
\newblock Stability conditions of generalised {Ising} spin glass models.
\newblock {\em Journal of Physics C: Solid State Physics 15}, 33 (Nov. 1982),
  L1187--L1193.

\bibitem{Leuzzi}
{\sc Leuzzi, L.}
\newblock Spin-glass model for inverse freezing.
\newblock {\em Philosophical Magazine 87}, 3-5 (2007), 543--551.

\bibitem{mottishaw_stability_1985}
{\sc Mottishaw, P.~J., and Sherrington, D.}
\newblock Stability of a crystal-field split spin glass.
\newblock {\em Journal of Physics C: Solid State Physics 18}, 26 (Sept. 1985),
  5201--5213.

\bibitem{panchenko_free_2005}
{\sc Panchenko, D.}
\newblock Free energy in the generalized {Sherrington}-{Kirkpatrick} mean field
  model.
\newblock {\em Reviews in Mathematical Physics 17}, 07 (Aug. 2005), 793--857.
\newblock arXiv: math/0405362.

\bibitem{sherrington_solvable_1975}
{\sc Sherrington, D., and Kirkpatrick, S.}
\newblock Solvable {Model} of a {Spin}-{Glass}.
\newblock {\em Physical Review Letters 35}, 26 (Dec. 1975), 1792--1796.

\bibitem{Subag}
{\sc Subag, E.}
\newblock The geometry of the {G}ibbs measure of pure spherical spin glasses.
\newblock {\em Invent. Math. 210}, 1 (2017), 135--209.

\bibitem{talagrand_mean_2011}
{\sc Talagrand, M.}
\newblock {\em Mean {Field} {Models} for {Spin} {Glasses} {Volume} {I}}.
\newblock Springer Berlin Heidelberg, Berlin, Heidelberg, 2011.

\bibitem{thouless_solution_1977}
{\sc Thouless, D.~J., Anderson, P.~W., and Palmer, R.~G.}
\newblock Solution of '{Solvable} model of a spin glass'.
\newblock {\em Philosophical Magazine 35}, 3 (Mar. 1977), 593--601.

\bibitem{yokota_first_order_1992}
{\sc Yokota, T.}
\newblock First-order transitions in an infinite-range spin-glass model.
\newblock {\em Journal of Physics: Condensed Matter 4}, 10 (Mar. 1992),
  2615--2622.

\end{thebibliography}

\end{document}